\documentclass{amsart}
\usepackage{latexsym, amscd, amsmath, amsthm, amssymb}
\usepackage{graphicx}
\usepackage[dvips,ps,all]{xy}
\xyoption{line}
\xyoption{color}
\xyoption{frame}

\usepackage{wrapfig}
\usepackage{multirow}

\newcommand{\LYYY}[5][{}]{\ensuremath{
 \xygraph{
!{<0pt,0pt>;<4pt,0pt>:<0pt,-4pt>::}
!{(1,4)}*{\scriptscriptstyle #1}
!{(1,3)}="a"
!{(1,1)}="b"
!{(-1,-1)}="c"
!{(3,-1)}="d"
!{(3,-2)}*{\scriptscriptstyle #5}
!{(-3,-3)}="e"
!{(1,-4)}*{\scriptscriptstyle #4}
!{(1,-3)}="f"
!{(-5,-5)}="g"
!{(-1,-5)}="h"
!{(-5,-6)}*{\scriptscriptstyle #2}
!{(-1,-6)}*{\scriptscriptstyle #3}
"b"-"c"
"b"-"d"
"c"-"e"
"c"-"f"
"e"-"g"
"e"-"h"
}
}}

\newcommand{\RYYLY}[5][{}]{\ensuremath{
 \xygraph{
!{<0pt,0pt>;<4pt,0pt>:<0pt,-4pt>::}
!{(1,4)}*{\scriptscriptstyle #1}
!{(1,3)}="a"
!{(1,1)}="b"
!{(-1,-1)}="c"
!{(3,-1)}="d"
!{(5,-3)}="e"
!{(1,-3)}="f"
!{(-1,-5)}="g"
!{(3,-5)}="h"
!{(-1,-2)}*{\scriptscriptstyle #2}
!{(-1,-6)}*{\scriptscriptstyle #3}
!{(3,-6)}*{\scriptscriptstyle #4}
!{(5,-4)}*{\scriptscriptstyle #5}
"b"-"c"
"b"-"d"
"d"-"e"
"d"-"f"
"f"-"g"
"f"-"h"
}
}}

\newtheorem{lemma}{Lemma}[section]

\newtheorem{cor}[lemma]{Corollary}
\newtheorem{thm}[lemma]{Theorem}

\def\sg{\sigma}
\def\S{\mathfrak{S}}

\DeclareMathOperator{\red}{red}

\DeclareMathOperator{\bl}{bl}

\DeclareMathOperator{\sh}{sh}

\def\LT{\mathcal{LT}}

\def\Q{\mathcal{Q}}
\newcommand{\egf}[3]{#1^{#2, #3}}
\newcommand{\coeff}[4]{#1^{#2, #3}_{#4}}
\newcommand{\egfk}[4]{#1^{#2, #3}_{#4}}

\newcommand{\av}[3]{#1_{#2}(#3)}

\newcommand{\seq}[1]{\mathbf{#1}}
\newcommand{\shift}[1]{\sh(\seq{#1})}
\newcommand{\vseq}[1]{\mathbf{\vec{#1}}}
\newcommand{\kshift}[2]{\sh_{#1}(\vseq{#2})}

\title{Block Patterns in Stirling Permutations}
\author{Jeffrey B. Remmel, Andrew Timothy Wilson}
\address{Department of Mathematics \newline \indent
University of California, San Diego \newline \indent
La Jolla, CA, 92093-0112, USA}
\email{jremmel@math.ucsd.edu, atwilson@math.ucsd.edu}
\thanks{The second author is partially supported by the Department of Defense (DoD) through the National Defense Science \& Engineering Graduate Fellowship (NDSEG) Program.}

\begin{document}
\begin{abstract}
We introduce and study a new notion of patterns in Stirling and $k$-Stirling permutations, which we call block patterns. We prove a general result which allows us to compute generating functions for the occurrences of various block patterns in terms of generating functions for the occurrences of patterns in permutations. This result yields a number of applications involving, among other things, Wilf equivalence of block patterns and a new interpretation of Bessel polynomials. We also show how to interpret our results for a certain class of labeled trees, which are in bijection with Stirling permutations. 
\end{abstract}

\keywords{Stirling permutations; permutation patterns; blocks; exponential generating functions; Bessel polynomials.}

\maketitle


%
%
%
%

\section{Introduction}
\label{sec:intro}

The set $\Q_{n}$ of \emph{Stirling permutations of order $n$} is the collection of all permutations of the multiset $\{1^2,2^2, \ldots, n^2\}$ such that every element between the two occurrences of $i$ is greater than $i$ for each $i \in \{1, 2, \ldots, n\}$. Stirling permutations were initially defined in \cite{stirling}.  Recently, there has been 
considerable work on Stirling permutations and their generalizations in, among others, \cite{B,J,generalized,blocks,HV}. 
These papers studied the distributions of various types of patterns in Stirling permutations and their generalizations. Moreover, there are natural bijections between Stirling 
permutations and various families of trees, as mentioned in \cite{J,generalized,dotsenko}. Thus, studying patterns in Stirling permutations is equivalent 
to studying patterns in these trees. We outline one of these bijections in Section \ref{sec:trees}. 

Each of the papers mentioned above also developed methods to deal with a common generalization of Stirling permutations known as $k$-Stirling permutations. The set $\Q_{n, k}$ of  \emph{$k$-Stirling permutations of order $n$} is the set of all rearrangements of the multiset $\{1^{k}, 2^{k}, \ldots, n^{k} \}$ such that every element between two consecutive occurrences of $i$ is greater than $i$ for each $i \in \{1, 2, \ldots, n\}$. Thus $\Q_{n,2} = \Q_n$. 

The main goal of this paper is to study a new type of pattern among 
blocks in Stirling permutations and $k$-Stirling permutations. We will deal only with Stirling permutations for now, leaving $k$-Stirling permutations for Section \ref{sec:k}. To begin, we must define blocks in Stirling permutations, essentially following \\ \cite{blocks}.
For $\sg \in \Q_{n}$, we  let $(i,i)_{\sg}$ denote 
the consecutive string of elements of $\sg$ between the two 
occurrences of $i$ in  $\sg$ and we let $[i,i]_{\sg}$ denote 
the consecutive string of elements of $\sg$ between and including the two 
occurrences of $i$ in $\sg$. We  shall write $[i,i]_{\sg} \subset[j,j]_\sg$ if $i \neq j$ and 
$[i,i]_\sg$ is a consecutive substring of $[j,j]_\sg$. 

For any word $w$ over the alphabet of positive integers $\mathbb{P}$, we say that the \emph{reduced form} of $w$, written $\red(w)$, is equal to the word obtained by replacing each of the occurrences of the $i$th smallest number in $w$ with the number $i$. Two words with the same reduced form are said to be \emph{order isomorphic}.

We say that 
$[i,i]_\sg$ is a level 1 block of $\sg$ if there is no 
$j$ such that $[i,i]_{\sg} \subset [j,j]_\sg$. For $\ell \geq 2$, 
we define the level $\ell$ blocks of $\sg$ inductively by saying 
that $[i,i]_\sg$ is a level $\ell$ block if there is a level 
$\ell-1$ block $[j,j]_\sg$ of $\sg$ such that $[i,i]_\sg \subset [j,j]_\sg$ 
and the reduced form $\red([i,i]_\sg)$ is a level 1 block in $\red((j,j)_\sg)$. If $[i,i]_{\sg} \subset [j,j]_\sg$, $[i,i]_{\sg}$ is a level $\ell$ block, 
and $[j,j]_{\sg}$ is a level $\ell-1$ block, then we will say that 
$[j, j]_\sg$ the \emph{parent} of the block $[i,i]_{\sg}$.

For example, if $ \sg = 4415778852213663$, then $[1,1]_\sg = 1577885221$ and
$(1,1)_\sg = 57788522$.
The level 1 blocks of $\sg$ are $[4,4]_\sg$, $[1,1]_\sg$,  and $[3,3]_\sg$, 
the level 2 blocks of $\sg$ are $[5,5]_\sg$, $[2,2]_\sg$, 
and $[6,6]_\sg$, and the level 3 blocks are $[7,7]_\sg$ and $[8,8]_\sg$.

We say that blocks $[i,i]_\sg$ and $[j,j]_\sg$ are {\em siblings} 
if either they are both level 1 blocks or they share the same parent. Returning to the example \\  $ \sg = 4415778852213663$, the only level 
2 blocks which are siblings are $[5,5]_\sg$ and $[2,2]_\sg$. 
Finally, the maximum level of any block in a Stirling permutation is the \emph{height} of that Stirling permutation. 

We will consider a \emph{permutation pattern} to be a permutation, i.e.\ an element of the symmetric group $\S_m$ for some $m \geq 1$, that may have some of its consecutive elements underlined\footnote{This is not the most general definition of a permutation pattern, but it will suit our purposes. For more general definitions, see \cite{kitaev}.}. A permutation pattern $p$ of length $m$ is said to \emph{occur} in a word $w$ of length $n$ if there exist $1 \leq i_{1} < \ldots < i_{m} \leq n$ such that 
\begin{itemize}
\item $\red(w_{i_{1}} w_{i_{2}} \ldots w_{i_{m}})$ is equal to the permutation obtained by removing the underlines from $p$, and 
\item if $p_{j}$ and $p_{j+1}$ are connected by an underline then $i_{j+1} = i_{j}+1$. 
\end{itemize}
In other words, the underlines insist that certain entries are consecutive in $w$.

Often, the word $w$ is also a permutation. For example, the pattern $2\underline{31}$ occurs twice in the permutation $7253146$, by taking the the subsequences consisting of 253 and 231, respectively. The subsequence 254 does not form an occurrence of $2\underline{31}$ because, although these entries are order isomorphic to $231$, the 5 and 4 do not appear consecutively in the permutation.

When $p$ has no underlines, it is known as a \emph{classical} pattern. When every element of $p$ is underlined, it is a \emph{consecutive} pattern. Following \cite{kitaev}, we let $p(\pi)$ denote the number of occurrences of the pattern $p$ in the permutation $\pi$.

If $p$ is a permutation pattern, we let 
$\av{\S}{n}{p} = \{\pi \in \S_n : p(\pi) = 0 \}$ denote the set 
of permutations of $\S_n$ which avoid $p$. Similarly, if 
$A$ is a set of permutation patterns, we let 
$\av{\S}{n}{A} = \{\pi \in \S_n : p(\pi) = 0 \ \forall p \in A\}$ denote 
the set of permutations of $\S_n$ which avoid all the patterns in 
$A$. We say that two patterns $p$ and $q$ are said to be \emph{Wilf equivalent} if $|\av{\S}{n}{p}| = |\av{\S}{n}{q}|$ for all $n \geq 1$. More generally, 
we say that two sets of patterns $P$ and $Q$ are said to be Wilf equivalent if $|\av{\S}{n}{P}| = |\av{\S}{n}{Q}|$ for all $n \geq 1$. 

If $A$ is a set of patterns we wish to avoid and $p$ is a pattern whose occurrences we want to count, then we shall consider the generating function 
\begin{align*}
\egf{F}{A}{p}(t, z) &= 1 + \sum_{n \geq 1} \frac{t^{n}}{n!} 
\coeff{f}{A}{p}{n}(z) 
\end{align*}
where
\begin{align*}
\coeff{f}{A}{p}{n}(z) &= \sum_{\pi \in \av{\S}{n}{A}} z^{p(\pi)}.
\end{align*}

We now have the terminology necessary to define block patterns. Given a permutation pattern $p$ of length $m$, we say $p$ occurs as a \emph{block pattern} in $\sg \in \Q_n$ if there exist blocks $[b_1, b_1]_{\sg}, \ldots, [b_m, b_m]_{\sg}$ appearing from left to right  in $\sg$ such that
\begin{itemize}
\item all the blocks $[b_1, b_1]_{\sg}, \ldots, [b_m, b_m]_{\sg}$ are siblings, 
\item $\red(b_1 b_2 \ldots b_m) = p$ when considering $p$ as a permutation, and
\item if $p_j$ and $p_{j+1}$ are connected by an underline then the second occurrence of the number $b_j$ in $\sg$ and the first occurrence of $b_{j+1}$ in $\sg$ are consecutive. 
\end{itemize}
Since the blocks are all siblings, they must all be of the same level in $\sg$. We say that this level is the \emph{level} of the occurrence of $p$ as a block pattern in $\sg$. We will write the number of occurrences of the pattern $p$ in $\sg \in \Q_n$ at level $\ell$ as $p^{(\ell)}(\sg)$ and the number of total occurrences of $p$ in $\sg$ as $p(\sg)$. 

For example, we consider the Stirling permutation $\sg = 4415778852213663$ and the pattern $p = 21$. We have 
\begin{itemize}
\item $p^{(1)}(\sg) = 2$ (with $b_1=4$, $b_2=1$ and $b_1=4$, $b_2=3$), and
\item $p^{(2)}(\sg) = 1$ (with $b_1=5$, $b_2=2$).
\end{itemize}
The entries 72 do not form an occurrence of 21 because the $7$ block is contained in the $5$ block, but the $2$ block is not, so the two blocks are not siblings.  If we instead consider the descent pattern $d = \underline{21}$, we get
\begin{itemize}
\item $d^{(1)}(\sg) = 1$ (with $b_1 = 4$, $b_2 = 1$), and
\item $d^{(2)}(\sg) = 1$ (with $b_1 = 5$, $b_2=2$).
\end{itemize}

We are now ready to define our main generating function. We will use boldface to indicate sequences, i.e.\ maps into the set of positive integers $\mathbb{P}$. Let $\seq{A} =(A_1,A_2, \ldots )$ be a sequence of sets of patterns and $\seq{p}=(p_1,p_2, \ldots )$ be a sequence of patterns.  We set 
\begin{align*}
\av{\Q}{n}{\seq{A}} &= \{\sg \in \Q_{n}: a^{(i)}(\sg) = 0 \text{ for all } a \in A_{i} \text{ and for any } i \in \mathbb{P} \}.
\end{align*}
Thus $\av{\Q}{n}{\seq{A}}$ is the set of Stirling permutations that 
avoid all the patterns in $A_i$ at level $i$ for all $i \geq 1$. 
Our main object of interest in this paper is 
the exponential generating function
\begin{align*}
\egf{G}{\seq{A}}{\seq{p}}(t; \seq{x}; \seq{y}) &= 1 + \sum_{n \geq 1} \frac{t^{n}}{n!} \coeff{g}{\seq{A}}{\seq{p}}{n}(\seq{x}; \seq{y})
\end{align*}
where 
\begin{align*}
\coeff{g}{\seq{A}}{\seq{p}}{n}(\seq{x}; \seq{y}) &= \sum_{\sg \in \av{\Q}{n}{\seq{A}}}  \prod_{i \geq 1} x_{i}^{p_i^{(i)}(\sg)} y_{i}^{\bl^{(i)}(\sg)}.
\end{align*}
Here for any $\sg \in Q_n$, $\bl^{(i)}(\sg)$ is the number of level $i$ blocks in $\sg$. Thus the generating function $\egf{G}{\seq{A}}{\seq{p}}(t; \seq{x}; \seq{y})$ keeps track of the number of occurrences of $p_i$ in 
the $i$th level of the permutations in $\av{\Q}{n}{\seq{A}}$. 

Given any sequence $\seq{s} =(s_1,s_2,s_3, \ldots)$, we let 
$\shift{s} = (s_2,s_3, \ldots )$. Thus $\shift{s}$ just 
removes the first element from the sequence. Our main theorem describes how to compute $\egf{G}{\seq{A}}{\seq{p}}$ if we already have $\egf{G}{\shift{A}}{\shift{p}}$ as well as the generating function $\egf{F}{A_1}{p_1}$. 

\begin{thm}
\label{thm:main}
\begin{align*}
\egf{G}{\seq{A}}{\seq{p}}(t; \seq{x}; \seq{y}) &= \egf{F}{A_1}{p_1} \left( y_1 \int_{0}^{t} \egf{G}{\shift{A}}{\shift{p}}(u; \shift{x}; \shift{y}) \mathrm{d}u, x_1 \right). 
\end{align*}
\end{thm}

The outline of this paper is as follows. In Section \ref{sec:thm} 
we prove Theorem \ref{thm:main} and derive several important 
corollaries of it. In Sections \ref{sec:height} and \ref{sec:higher}, we describe several 
special classes of sequences $\seq{A}$ and $\seq{p}$ where we can 
 explicitly compute $\egf{G}{\seq{A}}{\seq{p}}(t; \seq{x}; \seq{y})$.
Section \ref{sec:trees} contains a brief explanation of the relationship between our results and patterns in labeled trees. In Section \ref{sec:k}, we prove a generalization of Theorem \ref{thm:main} for $k$-Stirling 
permutations.

\section{Proof and Corollaries of Theorem \ref{thm:main}}
\label{sec:thm}

We start this section by giving a proof of Theorem \ref{thm:main}.
Our proof is similar to the proof of the compositional formula as described in Section 5.1 of \cite{stanley}. 
\begin{proof}
Let $\seq{A} =(A_1,A_2, \ldots )$ and 
$\seq{p}= (p_1,p_2, \ldots )$. For any exponential generating function $c(t) = \sum_{n \geq 0} c_n \frac{t^n}{n!}$, 
we let $c(t)|_{\frac{t^n}{n!}}$ denote $c_n$. 

We also let $\av{\Q^{m}}{n}{\seq{A}}$ denote the set of all permutations in $\av{\Q}{n}{\seq{A}}$ with exactly $b$ level 1 blocks. We describe a 
way in which we can uniquely construct all elements of $\av{\Q^{m}}{n}{\seq{A}}$. We will refer to this process as building a Stirling permutation by levels. 
\begin{enumerate}
\item First, we partition the set $\{1, 2, \ldots , n\}$ into an (unordered) collection of 
$m$ nonempty sets $\{S_1, S_2, \ldots, S_m\}$.
\item For each $i = 1 $ to $m$, we create a Stirling permutation 
in $\av{\Q}{|S_i| - 1}{\shift{A}}$ out of the non-minimal elements of $S_i$. 
That is, the reduced form of this object should be a member of $\Q_{|S_i| - 1}(\shift{A})$. 
We place the minimal element in $S_i$ before and after this (unreduced) Stirling permutation. We call the resulting (unreduced) 
Stirling permutation $\sg^{(i)}$.
\item The final Stirling permutation is the concatenation
$\sg^{(\pi_1)} \ldots \sg^{(\pi_m)}$ where $\pi = \pi_1 \ldots \pi_m$ 
is an element of $\av{\S}{m}{A_1}$.  
\end{enumerate}
By definition,
\begin{align}
\label{pf1}
\sum_{\sg \in \av{\Q^{m}}{n}{\seq{A}}} \prod_{i \geq 1} x_{i}^{p_{i}^{(i)}(\sg)} y_{i}^{\bl^{(i)}(\sg)} &= y_{1}^{m} \left. \egf{G}{\seq{A}}{\seq{p}}(t; \seq{x}; \seq{y}) \right|_{\frac{t^{n} y_{1}^{m}}{n!}} .
\end{align}
From the construction of Stirling permutations by levels, we see that (\ref{pf1}) equals
\begin{align*}
&\sum_{\pi \in \av{\Q}{m}{A_1}} x_1^{p_1(\pi)} y_1^m \times \\
&\sum_{\overset{a_1+ \cdots + a_m =n}{a_i \geq 1}} \binom{n}{a_1, \ldots, a_m} \prod_{i=1}^m
\left( \left. \egf{G}{\shift{A}}{\shift{p}}(t; \shift{x}; \shift{y}) \right|_{\frac{t^{a_i-1}}{(a_i-1)!}}\right)\\
= &\sum_{\pi \in \av{\Q}{m}{A_1}} x_1^{p_1(\pi)}y_1^m \times \\
&\sum_{\overset{a_1+ \cdots + a_m =n}{a_i \geq 1}} \binom{n}{a_1, \ldots, a_m} 
\prod_{i=1}^m \left( \left. \int_0^t \egf{G}{\shift{A}}{\shift{p}}(u; \shift{x}; \shift{y})\mathrm{d}u \right|_{\frac{t^{a_i}}{a_i!}}\right) \\
= &\sum_{\pi \in \av{\Q}{m}{A_1}} x_1^{p_1(\pi)} y_1^m 
\left. \left( \int_0^t \egf{G}{\shift{A}}{\shift{p}}(u; \shift{x}; \shift{y})\mathrm{d}u 
\right)^m \right|_{\frac{t^n}{n!}}. 
\end{align*}
Thus
\begin{align*}
&\sum_{\sg \in \av{\Q}{n}{\seq{A}}} \prod_{i \geq 1} x_{i}^{p_i^{(i)}(\sg)} y_{i}^{\bl^{(i)}(\sg)} \\
= &\sum_{m=1}^n \sum_{\sg \in \av{\Q^{m}}{n}{\seq{A}}} 
\prod_{i \geq 1} x_{i}^{p_i^{(i)}(\sg)} y_{i}^{\bl^{(i)}(\sg)} \\
= &\sum_{m=1}^n \sum_{\pi \in \av{\Q}{m}{A_1}} x_1^{p_1(\pi)}y_1^m
\left. \left(  \int_0^t \egf{G}{\shift{A}}{\shift{p}}(u; \shift{x}; \shift{y})\mathrm{d}u 
\right)^m \right|_{\frac{t^n}{n!}} \\
=  &\left. \egf{F}{A_1}{p_1} \left(y_1\int_0^t \egf{G}{\shift{A}}{\shift{p}}(u; \shift{x}; \shift{y})\mathrm{d}u,x_1\right) \right|_{\frac{t^n}{n!}}.
\end{align*}  
Finally, we have
\begin{align*} 
\egf{G}{\seq{A}}{\seq{p}}(t; \seq{x}; \seq{y})  &= 
1 + \sum_{n \geq 1} \frac{t^n}{n!} \sum_{\sg \in \av{\Q}{n}{\seq{A}}} \prod_{i \geq 1} x_{i}^{p_i^{(i)}(\sg)} y_{i}^{\bl^{(i)}(\sg)} \\
&=1 + \sum_{n \geq 1} \frac{t^n}{n!}  \left. \egf{F}{A_1}{p_1} \left( y_1\int_0^t \egf{G}{\shift{A}}{\shift{p}}(u; \shift{x}; \shift{y})\mathrm{d}u,x_1 \right) \right|_{\frac{t^n}{n!}} \\
&= \egf{F}{A_1}{p_1} \left( y_1 \int_{0}^{t} \egf{G}{\shift{A}}{\shift{p}}(u; \shift{x}; \shift{y}) \mathrm{d}u, x_1 \right)
\end{align*}
which proves Theorem \ref{thm:main}.
\end{proof}

We note that  Theorem \ref{thm:main} implies the following recursion for $\coeff{g}{\seq{A}}{\seq{p}}{n}(\seq{x}; \seq{y})$, which we recall is equal to $n!$ times the coefficient of $t^n$ in $\egf{G}{\seq{A}}{\seq{p}}(t; \seq{x}; \seq{y})$. For any partition $\lambda \vdash n$, let $m_i(\lambda)$ be the multiplicity of the number $i$ in $\lambda$ and $\ell(\lambda)$ be the length, i.e.\ number of nonzero parts, of $\lambda$.

\begin{cor}
\label{cor:rec}
\begin{align*}
\coeff{g}{\seq{A}}{\seq{p}}{n}(\seq{x}; \seq{y}) &= \sum_{\lambda \vdash n} \binom{n}{\lambda} \frac{ y_{1}^{\ell(\lambda)} \coeff{f}{A_1}{p_1}{\ell(\lambda)}(x_1) } { m_{1}(\lambda)! \ldots m_{n}(\lambda)!}  \prod_{i=1}^{\ell(\lambda)} \coeff{g}{\shift{A}}{\shift{p}}{\lambda_{i} - 1}(\shift{x}; \shift{y}).
\end{align*}
\end{cor}

This corollary follows immediately from taking coefficients in Theorem \ref{thm:main}. Although this recursion is still quite complicated, it is useful for generating small examples for sequences of patterns which do not have nice exponential generating functions. 

Next, we see that  Theorem \ref{thm:main} allows us to derive a statement about a type of Wilf equivalence for block patterns.

\begin{cor}
\label{cor:wilf}
Let $\seq{A}$ and $\seq{A^{\prime}}$ be sequences of sets of patterns and $\seq{p}$ and $\seq{p^{\prime}}$ be sequences of patterns. If  
\begin{align*}
\egf{G}{\seq{A}}{\seq{p}}(t; \seq{x}; \seq{y}) &= \egf{G}{\seq{A^{\prime}}}{\seq{p^{\prime}}}(t; \seq{x}; \seq{y}) 
\end{align*}
then $\egf{F}{A_i}{p_i}(t, z) =  \egf{F}{A^{\prime}_i}{p^{\prime}_i}(t, z)$ for all $i \in \mathbb{P}$. In the other direction, if we also assume that there exists some $r \geq 0$ such that 
\begin{align*}
\egf{G}{(A_r, A_{r+1}, \ldots)}{(p_r, p_{r+1}, \ldots)}(t; \seq{x}; \seq{y}) &= \egf{G}{(A^{\prime}_r, A^{\prime}_{r+1}, \ldots)}{(p^{\prime}_r, p^{\prime}_{r+1}, \ldots)}(t; \seq{x}; \seq{y})
\end{align*}
then the converse is true, i.e.\ $\egf{G}{\seq{A}}{\seq{p}}(t; \seq{x}; \seq{y}) = \egf{G}{\seq{A^{\prime}}}{\seq{p^{\prime}}}(t; \seq{x}; \seq{y}) $.
\end{cor}

%

\begin{proof}
We begin by assuming
\begin{align*}
\egf{G}{\seq{A}}{\seq{p}}(t; \seq{x}; \seq{y}) &= \egf{G}{\seq{A^{\prime}}}{\seq{p^{\prime}}}(t; \seq{x}; \seq{y}) .
\end{align*}
We will prove $\egf{F}{A_i}{p_i}(t, z) =  \egf{F}{A^{\prime}_i}{p^{\prime}_i}(t, z)$ for all $i$ by induction on $i$. 

For the base case $i = 1$, we set $y_j = 0$ for each $j \geq 2$. This yields 
\begin{align*}
\egf{G}{\seq{A}}{\seq{p}}(t; \seq{x}; y_1, 0, 0, \ldots) &= 1 + \sum_{n \geq 1} \frac{t^n}{n!} \sum_{\stackrel{\sg \in \av{\Q}{n}{\seq{A}}}{\text{height}(\sg) = 1}}x_1^{p_1^{(1)}(\sg)} y_1^{n}
&= \egf{F}{A_1}{p_1}(ty_1, x_1)
\end{align*}
because Stirling permutations of height 1 correspond to permutations with each number written twice consecutively. We can apply the same argument to $\egf{G}{\seq{A^{\prime}}}{\seq{p^{\prime}}}$ to finish the base case. 

For the induction step, we wish to prove $\egf{F}{A_i}{p_i} =  \egf{F}{A^{\prime}_i}{p^{\prime}_i}$ assuming this statement has been proven for each positive integer less than $i$.  We set $y_j = 0$ for $j > i$. The theorem states that $\egf{G}{\seq{A}}{\seq{p}}$ can be obtained by repeatedly integrating $\egf{F}{A_i}{p_i}$ and plugging that result into $\egf{F}{A_j}{p_j}$, where $j = i-1, i-2, \ldots, 1$. Since the two generating functions $\egf{G}{\seq{A}}{\seq{p}}$  and 
$\egf{G}{\seq{A}^{\prime}}{\seq{p}^{\prime}}$ are known to be equal, induction implies that we must have $\egf{F}{A_r}{p_r} =  \egf{F}{A^{\prime}_r}{p^{\prime}_r}$. 

In order to prove the other direction, we simply iterate the theorem $r$ times.
\end{proof}

Perhaps the most classical example of Wilf equivalence is the result of \\ \cite{knuth} that all classical patterns of length 3 are Wilf equivalent. Applying this fact to Corollary \ref{cor:wilf}, we learn that, if there exists an $i \geq 1$ such that $A_i = \{q\}$, where $q \in \S_3$ is a classical pattern of length 3, and $p_i = \emptyset$, then we can change $q$ to a different classical pattern of length 3 without altering $\egf{G}{A}{E}$. 

We say a consecutive pattern $q$ of length $m$ is \emph{minimally overlapping} if the shortest permutation that contains at least two copies of $q$ has length $2m-1$. In other words, no two occurrences of $q$ can overlap in more than one position. It was conjectured in \cite{elizalde} and later proven independently in \\ \cite{duane-remmel} and \cite{dotsenko-khoroshkin} that any two minimally overlapping patterns $q$ and $q^{\prime}$ of length $m$ are strongly Wilf equivalent, i.e.\ $\egf{F}{\{\emptyset\}}{q}(t, z) = \egf{F}{\{\emptyset\}}{q^{\prime}}(t, z)$, whenever $q_1 = q^{\prime}_1$ and $q_m = q^{\prime}_m$. Corollary \ref{cor:wilf} implies that, for any $i \geq 1$ such that $A_i = \emptyset$ and $p_i = q$, where $q$ is a minimally overlapping pattern of length $m$, we can change $q$ to a minimally overlapping pattern $q^{\prime}$ of length $m$ without changing $\egf{G}{\seq{A}}{\seq{p}}$ as long as $q_1 = q^{\prime}_1$ and $q_m = q^{\prime}_m$. 

In fact, essentially any result about Wilf equivalence of permutation patterns can be extended via Corollary \ref{cor:wilf} in this way. Some other examples of such results can be found in \cite{stankova, backelin-west-xin}. 

\section{Stirling Permutations of Restricted Height}
\label{sec:height}


In this section as well as the next, we derive a variety of generating functions from Theorem \ref{thm:main}. Some of these generating functions are well-studied, while others seem to be new and may be of interest in future work. In this section, we study Stirling permutations whose height is at most some fixed number. These permutations are especially nice for two reasons. First, Stirling permutations of height 1 correspond to permutations (with every entry written twice consecutively). This allows us to directly apply any known exponential generating function from the theory of permutation patterns. Second, our main theorem provides a way to understand Stirling permutations of height at most $h$ if we understand Stirling permutations of height at most $h-1$. In other words, if we have an exponential generating function for the permutation pattern case and we know how to integrate this generating function, we can provide closed-form generating functions for the restricted-height Stirling permutation case. Even if we cannot integrate the generating function, we can obtain initial terms for the sequence from the recursion in Corollary \ref{cor:rec}.

\subsection{Height $\leq 2$}

First we will deal with Stirling permutations of height at most 2. In order to obtain the class of Stirling permutations whose height is at most 2, we want to ``avoid'' the set of patterns $\{1\}$ at level 3. Since this pattern is unavoidable, the resulting class of permutations can only have blocks at levels 1 and 2. Thus throughout this section we will set $A_3 = \{1\}$.

These objects fit nicely into the context of the main theorem because the shifted pattern sequence $\shift{A} = (\emptyset, \{1\}, \emptyset, \ldots)$ implies that the generating function \\$\egf{G}{\shift{A}}{\shift{p}}$ is a sum over permutations. We will use this idea repeatedly in the remainder of this section to produce several examples.

\ \\
{\bf Example 1.} $\seq{A} = (\{21\}, \{\emptyset\}, \{1\}, \{\emptyset\}, \ldots )$ and $\seq{p} = (\emptyset, \emptyset, \ldots)$. \\

Any Stirling permutation that avoids $\seq{A}$ is equal to a series of blocks of height 1, written in order of increasing minimal element. These objects biject to permutations decomposed into cycles, so we should expect to see the unsigned Stirling numbers of the first kind. For example, $14422331577566$ is an example of such a Stirling permutation for $n = 7$. This corresponds to the permutation with cycle decomposition $(1, 4, 2, 3), (5, 7), (6)$. We have $\egf{F}{\{21\}}{\emptyset}(t, z) = \exp(t)$ and 
\begin{align*}
\int_{0}^{t} \egf{G}{\shift{A}}{\shift{p}}(u; \seq{x}; \seq{y}) \mathrm{d}u &= \int_{0}^{t} \frac{1}{1-uy_1} \mathrm{d} u \\
&= \frac{- \log(1 - ty_1)}{y_1}.
\end{align*}
Thus
\begin{align}
\label{stirling1}
\egf{G}{\seq{A}}{\seq{p}}(t; \seq{x}; \seq{y}) &= \exp \left( - \frac{y_1}{y_2} \log (1 - ty_2) \right) = (1 - ty_2)^{-y_{1}/y_{2}}
\end{align}
which, as we would expect from the above discussion, is equal to the exponential generating function
\begin{align*}
1 + \sum_{n \geq 1} \frac{t^{n}}{n!} \sum_{\pi \in \S_n} y_1^{\# \text{cycles in } \pi} y_2^{n - \# \text{ cycles in } \pi} .
\end{align*}
If we set $y_2 = 1$, we obtain the exponential generating function for the unsigned Stirling numbers of the first kind.

\ \\
{\bf Example 2.} $\seq{A} = (\{21\}, \{21\}, \{1\}, \{\emptyset\}, \ldots)$ 
and $\seq{p} = (\emptyset, \emptyset, \ldots )$. \\

In this case,  we obtain Stirling permutations whose blocks increase from left to right at both level 1 and level 2. One can see that the Stirling permutations $\sg \in \Q_n$ that avoid $\seq{A}$ correspond to partitions of the set $\{1, 2, \ldots, n\}$. Indeed, Theorem \ref{thm:main} gives
\begin{align}
\label{stirling2}
\egf{G}{\seq{A}}{\seq{p}}(t; \seq{x}; \seq{y}) &= \exp \left( \frac{y_1}{y_2} \left( \exp(ty_2) - 1 \right) \right)
\end{align}
which is equal to 
\begin{align*}
\sum_{n \geq 0} \frac{t^n}{n!} \sum_{\tau} y_{1}^{\# \text{ parts of } \tau} y_{2}^{n - \# \text{ parts of } \tau} 
\end{align*}
where the second sum is over partitions of the set $\{1, 2, \ldots n\}$. If we set $y_2 = 1$ and take $n!$ times the coefficient of $t^{n}y_{1}^{k}$ in this function, we obtain the triangle of Stirling numbers of the second kind. Through similar methods, we can find the ordered unsigned Stirling numbers of the first kind and the ordered Stirling numbers of the second kind by setting $\seq{A} = (\{\emptyset\}, \{\emptyset\}, \{1\}, \{\emptyset\}, \ldots)$ and $(\{\emptyset\}, \{21\}, \{1\}, \{\emptyset\}, \ldots)$, respectively.

Next we shall show how we can enumerate simple patterns at level 1. 

\ \\
{\bf Example 3.} $\seq{A} = (\emptyset, \emptyset, \{1\}, \emptyset, \ldots)$ and $\seq{p} = (\underline{21}, \emptyset, \ldots)$. \\

In this case we are counting block descents at level 1 while 
not avoiding any other patterns. 
We know that
\begin{align*}
\egf{F}{\emptyset}{\underline{21}}(t, z) &= \frac{z-1}{z - \exp(t(z-1))}
\end{align*}
and
\begin{align*}
\int_{0}^{t} \egf{G}{\shift{A}}{\shift{p}}(u; \seq{x}; \seq{y}) \mathrm{d} u &= \frac{- \log(1 - ty_1)}{y_1} .
\end{align*}
Plugging these into the main theorem, we obtain
\begin{align}
\label{stirling1des1}
 \egf{G}{\seq{A}}{\seq{p}}(t; \seq{x}; \seq{y}) &= \frac{x_{1}-1}{x_{1} - (1 - ty_{2})^{y_{1}(1-x_{1})/y_{2}}} .
\end{align}
This function refines the ordered unsigned Stirling numbers of the first kind, since setting $x_1 = 1$ yields the generating function for these numbers. We can think of the ordered unsigned Stirling numbers of the first kind as counting the number of ordered cycle decompositions. For example, in this setting $(5), (1, 4, 2), (3)$ and $(5), (3), (1, 4, 2)$ are counted separately. The generating function above enumerates descents (i.e.\ consecutive decreases) among minimal elements in cycles. In our examples, the first example has 1 descent and the second has 2. Although this is a classical application of the compositional formula for exponential generating functions, we are not aware of any work on these patterns.

\ \\
{\bf Example 4.} $\seq{A} = (\emptyset, \{21\},\{1\}, \emptyset, \ldots)$ and $\seq{p} = (\underline{21}, \emptyset, \ldots)$. \\

In this case we are counting block descents at level 1 while while 
insisting that the level two blocks in any level 1 block are increasing.   
In this case, $\egf{G}{\shift{A}}{\shift{p}}(u; \seq{x}; \seq{y})$ equals
$\exp(y_2t)$ so that 
\begin{align*}
\int_{0}^{t} \egf{G}{\shift{A}}{\shift{p}}(u; \seq{x}; \seq{y}) \mathrm{d} u &= \frac{\exp(y_2t)}{y2}.
\end{align*}
Hence  
\begin{align}
\label{stirling2des2}
\egf{G}{\seq{A}}{\seq{p}}(t; \seq{x}; \seq{y}) &= \frac{x_1 - 1}{x_1 - \exp \left( \frac{(x_1 - 1) y_1}{y_2}  ( \exp(ty_2) - 1) \right)} .
\end{align}
This is a refinement of the ordered Stiring numbers of the second kind. In particular, if we write an ordered set partition in the form $458|12|9|367$, using bars to separate parts, this function counts the number of descents between minimal elements. For our example, we would have 2 such descents. This function also does not seem to be studied in the literature.

\ \\
{\bf Example 5.} $\seq{A} = (\emptyset,\emptyset, \{1\}, \emptyset, \ldots )$ and 
$\seq{p} =(\underline{21},\underline{21}, \emptyset, \ldots )$.\\
 
In this case, $\egf{G}{\seq{A}}{\seq{p}}(t; \seq{x}; \seq{y})$ keeps track 
of the number of descents and blocks at level 1 and the number of 
descents and blocks at level 2 in Stirling permutation whose 
height is $\leq 2$. In this 
case, 
\begin{align*}
\egf{F}{\emptyset}{\underline{21}}(t, z) &= \frac{z-1}{\exp(t(z-1)) -z}, \\
\egf{G}{\shift{A}}{\shift{p}}(t; \seq{x}; \seq{y}) &= 
\frac{x_2-1}{\exp(y_2t(x_2-1)) -x_2}, 
\ \mbox{and}  \\
\int_{0}^{t} \egf{G}{\shift{A}}{\shift{p}}(u; \seq{x}; \seq{y}) \mathrm{d} u &=  \frac{y_2t(x_2-1)+\ln\left(\frac{1-x_2}{\exp(y_2t(x_2-1))-x_2}\right)}{x_2y_2}.
\end{align*}
This yields
\begin{align}
\label{des1des2}
\egf{G}{\seq{A}}{\seq{p}}(t; \seq{x}; \seq{y}) &= \frac{x_1 - 1}{x_1 - 
\exp \left(  \frac{y_1(x_1-1)}{x_2y_2} \left(y_2t(x_2-1)+\ln\left(\frac{1-x_2}{\exp(y_2t(x_2-1))-x_2}\right)  \right) \right)}.
\end{align}
This is a further refinement of the ordered unsigned Stirling numbers of the first kind. Here we are counting descents among minimal elements using $x_1$ and descents among non-minimal elements inside each cycle with $x_2$. 

\ \\
{\bf Example 6.} $\seq{A} = (\{\underline{321}\},\{21\}, \{1\}, \emptyset, \ldots )$ and 
$\seq{p} =(\underline{21},\emptyset, \emptyset, \ldots )$.\\

In this case, we are keeping track of block descents at level one while 
insisting that that the level 2 blocks in each level 1 block are 
increasing. To compute this generating function, we need a result from \cite{MenRem1}, namely that
\begin{align*}
F^{\{\underline{321}\},\underline{21}}(x,t) &= 
\frac{\exp(t/2)}{\cos(\frac{t\sqrt{4x-1}}{2}) - \frac{1}{\sqrt{4x-1}} \sin (\frac{t\sqrt{4x-1}}{2})}
\end{align*}
which is the generation function for distribution of 
descents in permutations that avoid $\underline{321}$. 

Thus 
\begin{align*}
\egf{G}{\seq{A}}{\seq{p}}(t; \seq{x}; \seq{y}) &= 
\frac{\exp(\exp(y_2t)/2y_2)}{\cos(\frac{\exp(y_2t)\sqrt{4x_1-1}}{2y_2}) - \frac{1}{\sqrt{4x_1-1}} \sin (\frac{\exp(y_2t)\sqrt{4x-1}}{2y_2})}.
\end{align*}

Of course, we should also consider patterns other than the descent pattern. Unfortunately, there are few exponential generating functions for enumeration of other patterns, so we will mostly deal with avoidance. We will set $\seq{p} = (\emptyset, \emptyset, \ldots)$, although in most of these examples we could enumerate descents at level 1 or level 2. If we set $\seq{A} = \{\{\underline{123}, \underline{321}\}, \{21\}, \{1\}, \{\emptyset\}, \ldots)$ we obtain objects that correspond to ordered partitions of $\{1, 2, \ldots, n\}$ whose minimal elements form a zigzag permutation. A classical result of \cite{andre} states that
\begin{align*}
\egf{F}{\{\underline{123}, \underline{321}\}}{\emptyset}(t, z) &= 2\sec(t) + 2\tan(t)
\end{align*}
so
\begin{align}
\label{zigzag}
\egf{G}{\seq{A}}{\seq{p}}(t; \seq{x}; \seq{y}) &= 2\sec \left( \frac{-y_1}{y_2} \log(1 - ty_2) \right)  + 2 \tan \left(\frac{-y_1}{y_2} \log(1 - ty_2) \right).
\end{align}

If we wish to switch $A_1$ and $A_2$ we just need to integrate Andre's generating function. For a more modern example, one could consult \\ \cite{elizalde-noy}, in which the authors obtained many exponential generating functions for permutations that avoid certain consecutive patterns.

\subsection{Height $\leq 3$}
\label{ssec:height3}

If we wish to look at the set of Stirling permutations of height at most 3, we set $A(4) = \{1\}$. Computing the generating function in this case involves one more integral than in the height $\leq 2$ setting, so it becomes less likely that we can derive a closed-form generating function. Considering the examples in the previous section, we can integrate (\ref{stirling1}) and (\ref{stirling2}). Thus we have closed-form generating functions for sequences like $\seq{A} = (A_1, \{21\}, \{\emptyset\}, \{1\}, \{\emptyset\}, \ldots)$ and $\seq{A} = (A_1, \{\emptyset\}, \{21\}, \{1\}, \{\emptyset\}, \ldots)$, assuming that each $p_i = \emptyset$ and the generating function $\egf{F}{A_1}{\emptyset}(t, z)$ is known. However, it seems as if the functions in (\ref{stirling1des1}), (\ref{stirling2des2}), (\ref{des1des2}), and (\ref{zigzag}) do not have nice integrals.

However, we can end with a simple example which showes the power 
of the techniques. For example, we can consider 
$$\cosh(t) = \sum_{n \geq 0} \frac{t^{2n}}{(2n)!}$$
as the generating function for increasing permutations of even length. 
Using the ideas of the proof of Theorem \ref{thm:main}, it is 
easy to see that 
$$\egf{F}{\{\emptyset\}}{\emptyset}\left( y_1\int_0^t \cosh(y_2u)\mathrm{d}u,x_1 \right) = 
\frac{1}{1-\frac{y_1}{y_2}\sinh(y_2t)}$$ is the generating function 
of $y_1^{\bl_1(\sg)}y_2^{\bl_2(\sg)}$ over all 
of Stirling permutations $\sg$ of height $\leq 2$ such that every 
block at level 1 contains an even number of level two blocks. 

Next, one can compute that 
\begin{align*}
H(t, y_2, y_3) & = \int_0^t \frac{1}{1-\frac{y_2}{y_3}\sinh(y_3u)} \mathrm{d}u \\
&= 
\frac{2\mbox{arctan}\left( \frac{y_2}{\sqrt{-(y_2^2+y_3^2)}}\right) +
2 \mbox{arctan}
\left( \frac{-y_2-\tanh(y_3t/2)}{\sqrt{-(y_2^2+y_3^2)}} \right)}{\sqrt{(y_2^2+y_3^2)}}.
\end{align*}  

It then follows that 
$$
\frac{1}{1-y_1H(t, y_2, y_3)}
$$ 
is the generating function of $y_1^{\bl_1(\sg)}y_2^{\bl_2(\sg)}y_3^{\bl_3(\sg)}$ over all 
of Stirling permutations $\sg$ of height $\leq 3$ such that every 
block at level 2 contains an even number of level three blocks and 
$$
\frac{1}{1-y_1\left( \frac{H(t, y_2, y_3)+H(-t, y_2, y_3)}{2}\right)}
$$ 
 is the generating function of $y_1^{\bl_1(\sg)}y_2^{\bl_2(\sg)}y_3^{\bl_3(\sg)}$ over all 
of Stirling permutations $\sg$ of height $\leq 3$ such that every 
block at level 2 contains an even number of level 3 blocks and 
every block at level 1 contains an even number of level 2 blocks.
 
Clearly many other examples of this type can be constructed 
where we specify the conditions of the allowable level 2 and level 3 blocks 
in Stirling permutations of height $\leq 3$.

\section{Ignoring Higher Blocks}
\label{sec:higher}
In this section, we no longer set a maximum height for the Stirling permutations that we will consider. Instead, we ``ignore'' all blocks above a certain level. More specifically, we insist that $A_i = \{\emptyset\}$, $p_i = \emptyset$, and we set $y_i =1$ for all $i$ larger than some fixed integer. These conditions work well with Theorem \ref{thm:main} because we already have an exponential generating function for the case where $A_i = p_i = \emptyset$ 
and $y_i =1$ for all $i$, namely the exponential generating function for 
$|\Q_n|$
\begin{align*}
\sum_{n \geq 0} |\Q_n| \frac{t^n}{n!} &= \frac{1}{\sqrt{1-2t}} .
\end{align*}
Furthermore, we can integrate this function
\begin{align*}
\int_{0}^{t} \frac{\mathrm{d}u}{\sqrt{1-2u}} &= 1 - \sqrt{1-2t} .
\end{align*}
Then Theorem \ref{thm:main} allows us to introduce patterns to avoid and count at level 1. As before, the main obstacles to any situation here are finding the exponential generating function for the permutation case and integrating this function.

\ \\
{\bf Example 1.} $\seq{A} = \{\{21\}, \emptyset, \ldots\}$, $\seq{p} = (\emptyset, \emptyset, \ldots )$, and $y_i =1$ for all $i \geq 2$. \\

In this situation, we can obtain the generating 
function of $y_1^{\bl_1(\sg)}$ over the set of Stirling permutations that are increasing at level 1. Since
\begin{align*}
\egf{F}{21}{\emptyset}(t) &= \exp(t)
\end{align*}
we have
\begin{align*}
\egf{G}{\seq{A}}{\seq{p}}(t; \seq{x}; \seq{y}) &= \exp \left( y_1(1 - \sqrt{1-2t}) \right).
\end{align*}
Surprisingly, this is exactly the exponential generating function for the modified Bessel polynomials! These polynomials were first defined in \\ \cite{bessel} and earned their name from a connection to Bessel functions. Following \cite{carlitz}, we define the modified Bessel polynomials
\begin{align*}
B_n(y) &= \sum_{k=1}^{n} \frac{(2n - k - 1)!}{2^{n-k}(n-k)!(k-1)!} y^{k} .
\end{align*}
In \cite{carlitz}, the author proved
\begin{align*}
\sum_{n \geq 0} \frac{t^{n}}{n!} B_{n}(y) = \exp\left(y(1 - \sqrt{1 - 2t})\right) .
\end{align*}
This shows that the number of $\sg \in \Q_n$ with $k$ level 1 blocks whose minimal elements increase from left to right is equal to $\frac{(2n - k - 1)!}{2^{n-k}(n-k)!(k-1)!}$. We do not know of any other proofs of this fact.

\ \\
{\bf Example 2.} $\seq{A} = (\emptyset, \emptyset, \emptyset, \ldots )$, $\seq{p} = (\underline{21}, \emptyset, \emptyset, \ldots)$, and 
$y_i =1$ for all $i \geq 2$. \\

By altering $A_1$ and $p_1$ we can obtain many other interesting generating functions. In this particular example, we  count occurrences of the descent pattern $\underline{21}$ at level 1 and make no restrictions at level 1. Using the exponential generating function for descents over permutations in Theorem \ref{thm:main} produces
\begin{align*}
\egf{G}{\seq{A}}{\seq{p}}(t; \seq{x}; \seq{y}) &= \frac{x_1 - 1}{x_1 - \exp \left( y_1(x_1-1) (1 - \sqrt{1-2t}) \right)} .
\end{align*}

\ \\
{\bf Example 3.} $\seq{A} = (\emptyset, \{21\}, \emptyset, \ldots )$, 
$\seq{p} = (\underline{21}, \emptyset, \emptyset, \ldots)$, and 
$y_i =1$ for all $i \geq 3$. \\

In this case, $\egf{G}{\seq{A}}{\seq{p}}(t; \seq{x}; \seq{y})$ is the 
generating function that keeps track of the number of 
descents and blocks at level 1 and the number of blocks at level 2 over 
the set of Stirling permutations whose level 2 blocks are increasing 
in each level 1 block. We obtain 
\begin{align*}
\egf{F}{\emptyset}{\underline{21}}(t, z) &= \frac{z-1}{\exp(t(z-1)) -z}, \\
\egf{G}{\shift{A}}{\shift{p}}(t; \seq{x}; \seq{y}) &= 
\exp (y_2(1-\sqrt{1-2u})), \ \mbox{and}  \\
\int_{0}^{t} \egf{G}{\shift{A}}{\shift{p}}(u; \seq{x}; \seq{y}) 
\mathrm{d}u \\
&=  \int_{0}^{t} \exp(y_2(1-\sqrt{1-2u}))\mathrm{d}u \\
&= \frac{-1-y_2+\exp(y_2(1-\sqrt{1-2u}))(1+y_2\sqrt{1-2t})}{y_2^2}. 
\end{align*}
This yields
\begin{align*}
\egf{G}{\seq{A}}{\seq{p}}(t; \seq{x}; \seq{y}) &= \frac{x_1 - 1}{x_1 - 
\exp \left(y_1(x_1-1) \frac{-1-y_2+\exp(y_2(1-\sqrt{1-2u}))(1+y_2\sqrt{1-2t})}{y_2^2}\right)}.
\end{align*}
As in Section \ref{sec:height}, we can attempt to integrate these functions again to introduce new pattern conditions at level 2.

\ \\
{\bf Example 4.} $\seq{A} = (\{21\}, \{21\}, \ldots)$, $\seq{p} = \{\emptyset, \emptyset, \ldots\}$, $x_i = y_i = 1$ for all $i \geq 1$. \\

As one last example, suppose that we want to compute the number 
of Stirling permutations where there are no block descents at 
any level and we set $x_i = y_i =1$ for all $i$. This does not strictly fit into the format of ``ignoring higher blocks,'' but we can still accomplish our goal. We want to find $\egf{G}{\seq{A}}{\seq{p}}(t; \seq{1}; \seq{1})$, where  $\seq{1} =(1,1,\ldots )$. Then 
if $\egf{G}{\seq{A}}{\seq{p}}(t; \seq{1}; \seq{1}) = 
\sum_{n \geq 0} g_n \frac{t^n}{n!}$, Theorem \ref{thm:main} implies 
that 
\begin{align}\label{receq}
\egf{G}{\seq{A}}{\seq{p}}(t; \seq{1}; \seq{1}) &= 
F^{21,\emptyset}
\left(\int_0^t \egf{G}{\seq{A}}{\seq{p}}(u; \seq{1}; \seq{1})\mathrm{d}u,1\right) \nonumber \\ 
&= \exp \left(  \int_{0}^{t} \egf{G}{\seq{A}}{\seq{p}}(u; \seq{1}; \seq{1})\mathrm{d}u \right).
\end{align}
Since $g_0 =1$, it  is easy to see  that equation (\ref{receq}) completely determines the sequence $g_0, g_1, \ldots $.  In fact,
$\egf{G}{\seq{A}}{\seq{p}}(t; \seq{1}; \seq{1}) = \frac{1}{1-t}$ is 
a solution to this equation we must have $g_n = n!$ for all $n$. This is, of course, easy to see combinatorially. 
That is, if $\sg \in \Q_{n-1}$ has no block descents, then 
we can either insert the two copies of $n$ as a level 1 block at the right end of $\sg$ or immediately before the second occurrence of 
$i$ for any $i =1, \ldots , n-1$. 

\section{Labeled Trees}
\label{sec:trees}
In this section, we shall show that the results of the previous sections 
can be described in terms of patterns in trees. 

We define $\LT_{n}$ to be the set of planar, rooted, binary trees with $n+1$ leaves such that
\begin{enumerate}
\item each $i \in \{0, 1, 2, \ldots, n\}$ is used to label exactly one leaf, and
\item for any vertex, the smallest label used in its left subtree is less than the smallest label used in its right subtree. 
\end{enumerate}
For example
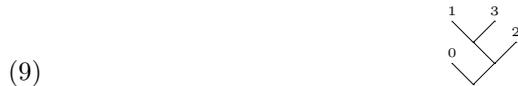
\begin{align}
\label{lt-example}
\RYYLY{0}{1}{3}{2}
\end{align}
is in $\LT_{3}$. However, the tree
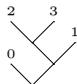
\begin{align*}
\RYYLY{0}{2}{3}{1}
\end{align*}
is not in $\LT_{3}$ because it fails the second condition at the vertex whose right child is the leaf labeled 1. 

A \emph{left} (respectively \emph{right}) \emph{comb} is a tree in which every right (respectively left) child is a leaf. In other words, a left comb only grows to the left, and a right comb only grows to the right. Consider a tree $T$ that satisfies all of the properties necessary to be in $\mathcal{LT}_{n}$ except that its leaves are numbered bijectively with some other $n$-element subset of the integers. We say that the $\emph{reduced form}$ of $T$, written $\red(T)$, is the unique the $S$ in $\mathcal{LT}_{n}$ obtained by replacing the $i$th smallest label in $T$ with $i$ for each $i \in \{1, \ldots, n\}$.

As outlined in \cite{dotsenko}, we can recursively define a bijection
\begin{align*}
\Phi : \mathcal{LT}_{n+1} \rightarrow \Q_{n} 
\end{align*}
by thinking of a tree $T$ as a left comb with subtrees $T_{1}, \ldots, T_{k}$ as its set of right children. We then apply $\Phi$ to each subtree $T_{i}$ after reducing $T_{i}$. We let $\Phi$ take the one-node tree to the empty word, and then let
\begin{align*}
\Phi(T) = \left\{ \begin{array}{ll}
1 \Phi(T_{1}) 1 & \text{if } k=1 \\
\red(\Phi(T_{1}) \ldots \Phi(T_{k})) & \text{if } k>1
\end{array} \right.
\end{align*}
Intuitively, we begin at the root and perform a depth-first search, exploring left as far as possible before exploring right. Each time we descend to the right, we record the smallest leaf label in that right subtree. When we ascend an edge that prompted us to record an $i$ when we descended it initially, we record the $i$ a second time. For example, $\Phi$ maps 
\begin{align}
\label{phi}
\RYYLY{0}{1}{3}{2} \longmapsto 133221 .
\end{align}

This bijection allows us to map the notions of level and block patterns from $Q_n$ to $\LT_n$. In particular, the \emph{level} of a vertex in a tree $T \in \LT_n$ is equal to 1 greater than the number of right branches on the path from the root to the vertex. Occurrences of the block pattern $p$ of length $m$ in a tree $T \in \LT_n$ correspond to appearances of left combs with labels $p_1, \ldots, p_m$ inside the tree $T$. The underlines of $p$ tell us which left branches can be removed in order to obtain the left comb inside $T$.

For example, the tree
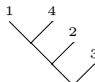
\begin{align*}
\LYYY{1}{4}{2}{3}
\end{align*}
has height 1. It has 3 blocks, all at level 1. Furthermore, it has two occurrences of the pattern $21$ at level 1, in the entries $4, 2$ and $4, 3$. It has only one occurrence of the descent pattern $\underline{21}$ at level 1, in the entries $4, 2$. This is because the leaves labeled 4 and 3 are not right children of consecutive nodes in the tree. 


When $p$ is a consecutive pattern, block patterns correspond exactly to the notion of consecutive patterns studied in \cite{dotsenko}. In that paper, the author computed the number of trees avoiding many small sets of consecutive tree patterns and proved a general asymptotic result for these consecutive patterns.

\section{$k$-Stirling Permutations}
\label{sec:k}

In this section, we shall prove an analog of Theorem \ref{thm:main} 
for $k$-Stirling permutations. This theorem will reduce to Theorem \ref{thm:main} when $k=2$.

For $\sg \in Q_{n,k}$, we let 
$[i,i]_\sg$ denote the consecutive segment of $\sg$ that lies 
between the first occurrence of $i$ in $\sg$ and the last occurrence 
of $i$ in $\sg$. We let $(i,i)_\sg$ be the word that results 
by removing all occurrence of $i$ from $[i,i]_\sg$. 
For $i \neq j$, we write $[i,i]_\sg \subset [j,j]_\sg$ 
if  $[i,i]_\sg$ is a consecutive subsequence of $[j,j]_\sg$.  
We  say that $[i,i]_{\sg}$ is a level 1 block if it is not contained in any $[j,j]_{\sg}$ for $j \neq i$. For $\ell \geq 2$, 
we define the level $\ell$ blocks of $\sg$ inductively by saying 
that $[i,i]_\sg$ is level $\ell$ block if there is a level 
$\ell-1$ block $[j,j]_\sg$ of $\sg$ such that $[i,i]_\sg \subset [j,j]_\sg$ 
and the reduced form $\red([i,i]_\sg)$ is a level 1 block in $\red((j,j)_\sg)$. If $[i,i]_{\sg} \subset [j,j]_\sg$, $[i,i]_{\sg}$ is level $k$ block, 
and $[j,j]_{\sg}$ is a level $k-1$ block, then we will say that 
$[j, j]_\sg$ the \emph{parent} of the block $[i,i]_{\sg}$.

Notice that, for any $k \geq 3$, the level $k$ blocks contained 
in a given level $k-1$ block $[h,h]_\sg$ in $\sg \in \Q_{n, k}$ 
are naturally partitioned 
into $k-1$ groups depending on which two consecutive occurrences of 
$h$ the blocks fall between. We say that the level $k$ block is of \emph{type} $s$ if it occurs between the $s$th and $s+1$st occurrences of $h$. Then 
we say that two blocks are \emph{siblings} if they are both level 1 blocks or if they share the same parent $[j,j]_{\sg}$ and are of the same type.

We let $p{(\ell, s)}(\sg)$ equal the number of occurrences of the block pattern $p$ at level $\ell$ and type $s$. Notice that, since level 1 blocks do no have a type, this definition only makes sense for $\ell \geq 2$. Now, instead of avoiding sequences of sets, we can avoid sequences of tuples of sets. We will write $\vseq{A}$ for a sequence $(A_1, A_2, \ldots)$ such that
\begin{itemize}
\item $A_1$ is a set of patterns
\item $A_{i}$ is a $(k-1)$-tuple of sets of patterns, which we will write as \\ $(A_{i, 1}, \ldots, A_{i, k-1})$. 
\end{itemize}
Similarly, $\vseq{p}$ indicates a sequence whose first entry is a pattern and whose other entries are $(k-1)$-tuples of patterns. We suppress $k$ from the notation. Then we can write $\av{\Q}{n, k}{\vseq{A}}$ for the set of $k$-Stirling permutations of order $n$ that avoid each pattern in $A_1$ at level 1 and each pattern in $A_{i, j}$ among the blocks at level $i$ and type $j$. By $\vseq{x}$ we denote the set of variables 
\begin{align*}
x_1, x_{2, 1}, \ldots, x_{2, k-1}, x_{3, 1}, \ldots, x_{3, k-1}, \ldots .
\end{align*}
Now our main generating function is
\begin{align*}
\egfk{G}{\vseq{A}}{\vseq{p}}{k}(t; \vseq{x}; \vseq{y}) &= 1 + \sum_{n \geq 1} \frac{t^{n}}{n!} \sum_{\sg \in \av{\Q}{n, k}{\vseq{A}}}  x_{1}^{p_{1}^{(1)}(\sg)} y_{1}^{\bl^{(1)}(\sg)} \prod_{i \geq 2} \prod_{j=1}^{k-1} x_{i, j}^{p_{i, j}^{(i, j)}(\sg)} y_{i, j}^{\bl^{(i, j)}(\sg)} 
\end{align*}
where $\bl^{(i, j)}(\sg)$ is the number of level $i$ blocks of type $j$ in $\sg$. 

Before we can state our theorem, we also need to refine the operator $\sh$. Namely, if 
$\seq{\tilde{s}}= (s_1,(s_{2,1}, \ldots, s_{2,k-1}), (s_{3,1}, \ldots, s_{3,k-1}), \ldots $, then by $\kshift{j}{A}$, we mean the new sequence of tuples
$$\kshift{j}{A} = s_{2, j}, (s_{3, 1}, \ldots, s_{3, k-1}), \ldots .$$
 We now state a more general form of Theorem \ref{thm:main}.

\begin{thm}
\label{thm:k}
$\egfk{G}{\vseq{A}}{\vseq{p}}{k}(t; \vseq{x}; \vseq{y})$ is equal to
\begin{align*}
\egf{F}{A_1}{p_1} \left( y_1 \int_{0}^{t}  \left (\prod_{j=1}^{k-1} \egfk{G}{\kshift{j}{A}}{\kshift{j}{A}}{k}(u; \kshift{j}{A}; \kshift{j}{A})\right) \mathrm{d}u, x_1 \right) 
\end{align*} 
\end{thm}

\begin{proof}
We set $\av{\Q^{m}}{n, k}{\vseq{A}}$ to be the set of all $\sg \in \av{\Q}{n, k}{\vseq{A}}$ with $m$ level 1 blocks. We describe a 
way in which we can uniquely construct all $k$-Stirling permutations 
$\sg \in \av{\Q^{m}}{n, k}{\vseq{A}}$. We will refer to this process as building a 
$k$-Stirling permutation by levels. 
\begin{enumerate}
\item First, we partition the set $\{1, 2, \ldots , n\}$ into an (unordered) collection of 
$m$ nonempty sets $\{S_1, S_2, \ldots, S_m\}$. Let $a_i$ be the minimal element 
of $S_i$ for $i=1, \ldots m$. 
\item  For each $i = 1 $ to $m$, we further partition 
each $S_i-\{a_i\}$ into a $k-1$ tuple of sets $(T_{i,1}, \ldots, T_{i,k-1})$, some of which may be empty 
.  Then we create a $k$-Stirling permutation 
in $Q_{|T_{i,j}|}(\kshift{j}{A})$ out of the elements of $T_{i,j}$. 
That is, the reduced form of this Stirling permutation should be a member of 
 $Q_{|T_{i,j}|}(\kshift{j}{A})$. Call this permutation 
$\sg^{(i,j)}$. If $T_{i,j}$ is empty, 
then $\sg^{(i,j)}$ is the empty permutation. We let 
$$\sg^{(i)} = a_i \sg^{(i,1)} a_i \sg^{(i,2)} a_i \ldots a_i \sg^{(i,k-1)} 
m_i.$$ 
\item We set the final $k$-Stirling permutation to be the concatenation \\
$\sg^{(\tau_1)} \ldots \sg^{(\tau_m)}$ where $\tau = \tau_1 \ldots \tau_m$ 
is an element of $\S_m$ which avoids $A_1$. 
\end{enumerate}

This process yields
\begin{align*}
&\sum_{\sg \in \av{\Q^{m}}{n, k}{\vseq{A}}} x_{1}^{p_{1}^{(1)}(\sg)} y_{1}^{\bl^{(1)}(\sg)} \prod_{i \geq 2} \prod_{j=1}^{k-1} x_{i, j}^{p_{i, j}^{(i, j)}(\sg)} y_{i, j}^{\bl^{(i, j)}(\sg)}\\
= &\sum_{\tau \in \av{\Q}{m}{A_1}} x_1^{p_1(\tau)} y_1^m \sum_{\overset{a_1+ \cdots + a_r =n}{a_i \geq 1}} 
\binom{n}{a_1, \ldots, a_r} \times \\ 
&\prod_{i=1}^{k-1} \sum_{\overset{b_{i,1}+ \cdots + b_{i,k-1}=a_i-1}{b_{i,j} \geq 0}} \binom{a_i-1}{b_{i,1}, \ldots, b_{i,k-1}} \times \\
&\prod_{j =1}^{k-1}  
\left(  \left. \egf{G}{\kshift{j}{A}}{\kshift{j}{p}}(t; \kshift{j}{x}; \kshift{j}{y}) \right|_{\frac{t^{b_{i,j}}}{b_{i,j}!}}\right)\\
= &\sum_{\tau \in \av{\Q}{m}{A_1}} x_1^{p_1(\tau)} y_1^m \sum_{\overset{a_1+ \cdots + a_r =n}{a_i \geq 1}}  \binom{n}{a_1, \ldots, a_r} \times \\
&\left. \left(\prod_{j =1}^{k-1}  
\left(  \egf{G}{\kshift{j}{A}}{\kshift{j}{p}}(t; \kshift{j}{x}; \kshift{j}{y})
\right)\right)\right|_{\frac{t^{a_i-1}}{(a_i-1)!}}\\
= &\sum_{\tau \in \av{\Q}{m}{A_1}} x_1^{p_1(\tau)} y_1^m \sum_{\overset{a_1+ \cdots + a_r =n}{a_i \geq 1}}  \binom{n}{a_1, \ldots, a_r} \times \\
&\int_0^t \left. \left(\prod_{j =1}^{k-1}  
\left(  \egf{G}{\kshift{j}{A}}{\kshift{j}{p}}(u; \kshift{j}{x}; \kshift{j}{y})
\right)\right)\right|_{\frac{t^{a_i}}{a_i!}} \\
=&\sum_{\tau \in \av{\Q}{m}{A_1}} x_1^{p_1(\tau)} y_1^m \left. \left(\int_0^t \left(\prod_{j =1}^{k-1}  
\left(  \egf{G}{\kshift{j}{A}}{\kshift{j}{p}}(u; \kshift{j}{x}; \kshift{j}{y})
\right)\right)\right)^m \right|_{\frac{t^{n}}{n!}}.
\end{align*}
Thus 
\begin{align*}
&\sum_{\sg \in \av{\Q}{n}{\vseq{A}}} x_{1}^{p_{1}^{(1)}(\sg)} y_{1}^{\bl^{(1)}(\sg)} \prod_{i \geq 2} \prod_{j=1}^{k-1} x_{i, j}^{p_{i, j}^{(i, j)}(\sg)} y_{i, j}^{\bl^{(i, j)}(\sg)}\\
= &\sum_{m=1}^n \sum_{\sg \in \av{\Q^{m}}{n, k}{\vseq{A}}} x_{1}^{p_{1}^{(1)}(\sg)} y_{1}^{\bl^{(1)}(\sg)} \prod_{i \geq 2} \prod_{j=1}^{k-1} x_{i, j}^{p_{i, j}^{(i, j)}(\sg)} y_{i, j}^{\bl^{(i, j)}(\sg)}
\\
= &\sum_{m=1}^n \sum_{\tau \in \av{\Q}{m}{A_1}} x_1^{p_1(\tau)} y_1^m \times \\
&\left. \left(\int_0^t \left(\prod_{j =1}^{k-1}  
\left(  \egf{G}{\kshift{j}{A}}{\kshift{j}{p}}(u; \kshift{j}{x}; \kshift{j}{y})
\right)\right)\right)^m \right|_{\frac{t^{n}}{n!}} \\
= & \left. \egf{F}{A_1}{p_1} \left(y_1 \left(\int_0^t \left(\prod_{j =1}^{k-1}  
\left(  \egf{G}{\kshift{j}{A}}{\kshift{j}{p}}(u; \kshift{j}{x}; \kshift{j}{y})
\right)\right)\right),x_1\right) \right|_{\frac{t^n}{n!}}.
\end{align*}  
Finally, we see that $\egf{G}{\seq{A}}{\seq{p}}(t; \seq{x}; \seq{y})$ equals
\begin{align*}
&= 1+ \sum_{n \geq 1} \frac{t^n}{n!} \sum_{\sg \in \av{\Q}{n}{\vseq{A}}} x_{1}^{p_{1}^{(1)}(\sg)} y_{1}^{\bl^{(1)}(\sg)} \prod_{i \geq 2} \prod_{j=1}^{k-1} x_{i, j}^{p_{i, j}^{(i, j)}(\sg)} y_{i, j}^{\bl^{(i, j)}(\sg)} \\
&= \egf{F}{A_1}{p_1} \left(y_1 \left(\int_0^t \left(\prod_{j =1}^{k-1}  
\left(  \egf{G}{\kshift{j}{A}}{\kshift{j}{p}}(u; \kshift{j}{x}; \kshift{j}{y})
\right)\right)\right),x_1\right)
\end{align*}
which proves Theorem \ref{thm:k}.
\end{proof}

As in the $k=2$ case, this theorem gives us a recursion for $n!$ times the coefficient of $t^n$ in $\egfk{G}{\seq{A}}{\seq{p}}{k}(t; \seq{x}; \seq{y})$. Since this recursion is rather unwieldy, we will not record it here.

Although it is difficult to find applications that use the full generality of Theorem \ref{thm:k} in which the integral is computable, we can compute the integral in some simple cases. For example, if we set $A_{3, j} = \{1\}$ for each $j$ then we obtain Stirling permutations of height at most 2. With no additional restrictions or enumeration at levels 2 or greater, we can set $k=2$ and compute the integral
\begin{align*}
\int_{0}^{t} \frac{\mathrm{d}u}{(1 - uy_{2, 1})(1 - uy_{2, 2})} &= \frac{1}{y_{2, 1} - y_{2, 2}} \log \left( \frac{ty_{2, 2} -1}{ty_{2, 1}} \right) .
\end{align*}
We could then plug this function into the $F$ corresponding to $A_1$ and $p_1$, as indicated in Theorem \ref{thm:k}. This integral remains computable for slightly larger values of $k$, although the result gets more and more complicated. An easier case is when $A_{3, j} = \{1\}$ and $A_{2, j} = \{21\}$ for all $j$, since the resulting integrand is 
\begin{align*}
\exp \left( \sum_{j=1}^{k-1} uy_{2, j} \right) .
\end{align*}

If we are willing to set some of the variables equal, we can obtain more closed-form generating functions. For example, if we set $A_{3, j} = \{1\}$, $p_{2, j} = \underline{21}$, $x_{2, j} = x_{2}$, and $y_{2, j} = y_{2}$ for all $j$, then our goal is to compute the integral
\begin{align*}
\int_{0}^{t} \left( \frac{1-x_2}{\exp \left( uy_{2}(x_{2}-1) \right) - x_{2} } \right)^{k} \mathrm{d}u
\end{align*}
which can be done for small values of $k$.

Finally, we can construct many examples where we restrict the 
possible size of blocks of various types. For example, 
we know that 
$$\cosh(y_{2,1}t) = \sum_{n \geq 0} \frac{y_{2,1}^{2n}t^{2n}n}{(2n)!}$$ is 
the generating function of even length permutations and 
$$\sinh(y_{2,2}t) = \sum_{n \geq 0} \frac{y_{2,1}^{2n+1}nt^{2n+1}}{(2n+1)n!}$$ 
is odd length permutations.
Then 
\begin{multline*}
\int_0^t \cosh(y_{2,1}u)\sinh(y_{2,2}u)\mathrm{d}u = \\
\frac{y_{2,2}-y_{2,2}\cosh(y_{2,1}t)\cosh(y_{2,2}t)+y_{2,1}\sinh(y_{2,1}t)
\sinh(y_{2,2}t)}{y_{2,1}^2 -y_{2,2}^2}.
\end{multline*}
By Theorem \ref{thm:k}, the generating function of
$y_1^{\bl_1(\sg)} y_{2,1}^{\bl_{2,1}(\sg)}y_{2,2}^{\bl_{2,2}(\sg)}$ 
over all $3$-Stirling permutation of height $\leq 2$ such 
that for any level one block $[j,j]$ of $\sg$,  its  
type one subblock of level 2 has even length and its type 2 level subblock 
is of odd length equals
$$\frac{1}{1 -y_1 \left( \frac{y_{2,2}-y_{2,2}\cosh(y_{2,1}t)\cosh(y_{2,2}t)+y_{2,1}\sinh(y_{2,1}t)
\sinh(y_{2,2}t)}{y_{2,1}^2 -y_{2,2}^2}\right)}.$$


\bibliographystyle{apalike}
\bibliography{abstract}

\end{document}